%% file: main.tex
\documentclass[conference]{IEEEtran}

\usepackage{amsmath,amssymb,amsfonts}
\usepackage{amsthm}
\usepackage{bm}             
\usepackage[ruled,vlined]{algorithm2e}
\usepackage{graphicx}
\usepackage{textcomp}
\usepackage{xcolor}
\usepackage{graphicx}
\usepackage{textcomp}
\usepackage{xcolor}
\usepackage{subfig}

\usepackage{xr}
\usepackage{tikz}
\usepackage{pgfplots}
\usepackage[numbers]{natbib}

\input{symbols.tex}

\newtheorem{theorem}{Theorem}

\newtheorem{remark}{Remark}

\newtheorem{lemma}{Lemma}
\newtheorem{corollary}{Corollary}
\newtheorem{definition}{Definition}

\newtheorem{proposition}{Proposition}

\def\BibTeX{{\rm B\kern-.05em{\sc i\kern-.025em b}\kern-.08em
    T\kern-.1667em\lower.7ex\hbox{E}\kern-.125emX}}

\begin{document}

\title{A Precise Performance Analysis of the Randomized Singular Value Decomposition\\}

\author{\IEEEauthorblockN{Danil Akhtiamov$^*$}
\IEEEauthorblockA{
\textit{Department of }\\
\textit{Computing and Mathematical Sciences} \\
\textit{Caltech, Pasadena, CA} \\
dakhtiam@caltech.edu}
\and
\IEEEauthorblockN{Reza Ghane$^*$}
\IEEEauthorblockA{
\textit{Department of }\\
\textit{ Electrical Engineering} \\
\textit{Caltech, Pasadena, CA}\\
rghanekh@caltech.edu}
\and
\IEEEauthorblockN{Babak Hassibi}
\IEEEauthorblockA{
\textit{Department of }\\
\textit{Electrical Engineering}  \\
\textit{Caltech, Pasadena, CA}\\
hassibi@caltech.edu}}

\maketitle

\def\thefootnote{*}\footnotetext{Equal contribution}

\begin{abstract}
The Randomized Singular Value Decomposition (RSVD) is a widely used algorithm for efficiently computing low-rank approximations of large matrices, without the need to construct a full-blown SVD. Of interest, of course, is the approximation error of RSVD compared to the optimal low-rank approximation error obtained from the SVD. While the literature provides various upper and lower error bounds for RSVD, in this paper we derive precise asymptotic expressions that characterize its approximation error as the matrix dimensions grow to infinity. Our expressions depend only on the singular values of the matrix, and we evaluate them for two important matrix ensembles: those with power law and bilevel singular value distributions. Our results aim to quantify the gap between the existing theoretical bounds and the actual performance of RSVD. Furthermore, we extend our analysis to polynomial-filtered RSVD, deriving performance characterizations that provide insights into optimal filter selection.
\end{abstract}

\section{Introduction}

The task of finding a low-rank approximation of a matrix lies at the cornerstone of many computational problems in data science and signal processing. For example, principal component analysis  \cite{jolliffe2002principal}, latent semantic indexing \cite{berry1995using} and collaborative filtering \cite{koren2009matrix} all fundamentally rely on extracting the dominant singular values and vectors from large-scale matrices. However, as modern datasets grow to unprecedented scales, traditional deterministic algorithms for computing the Singular Value Decomposition (SVD) face prohibitive computational and memory constraints. Randomized Singular Value Decomposition (RSVD) \cite{halko2011finding, tropp2023randomized} has emerged as a powerful alternative, leveraging randomized sampling techniques to compute decent low-rank approximations with dramatically reduced computational complexity. By exploiting the inherent low-rank structure present in many real-world matrices, RSVD achieves approximation errors comparable to deterministic methods, while reducing the time complexity from $O(dn \min(d,n))$ to $O(dnk)$, where $n$ and $d$ are the dimensions of the matrix, and $k$ is the rank of the approximation. In this work, we develop a novel analysis of RSVD performance and its polynomial-filtered generalizations using Gaussian Comparison Inequalities \cite{Gordon1985SomeIF, Gordon88}. We derive exact asymptotic expressions for the approximation error that enable us to quantify the tightness of existing theoretical bounds and provide insights into optimal filter selection. Our precise characterizations not only illuminate the strengths and limitations of current theoretical guarantees but also enable us to improve upon existing upper bounds. By comparing our results with bounds from the literature, we provide a more complete picture of RSVD behavior in practice.

\section{Preliminaries}

\subsection{Randomized SVD}

% Assume $\bA \in \bbR^{n \times d}$ is a matrix and our goal is to find another matrix $\hA \in \bbR^{n \times d}$ satisfying $\rk{\hA} \le k$ for a given $k \le \rk{\bA}$. A popular approach is the RSVD algorithm presented in \cite{halko2011finding}, which works as follows. 

The RSVD algorithm, presented in \cite{halko2011finding}, relies on the idea of Gaussian sketching and works as follows:

% \begin{algorithmic}\label{alg:randomized_approx}
% \REQUIRE Matrix $\bA \in \bbR^{n \times d}$, $k \le \rk{\bA}$, a probability distribution $\bbP$
% \ENSURE Approximation $\hat{\bA} \in \bbR^{n \times d}$ with $\rk{\hat{\bA}} \le k$
% \STATE Sample $\bOmega \in \bbR^{d \times k}$ i.i.d. from $\bbP$
% \STATE Define $\bY := \bA\bOmega$
% \STATE Decompose $\bY = \bQ\bR$, where $\bQ \in \bbR^{n \times k}$ satisfies $\bQ^T\bQ = \bI_k$ and $\bR \in \bbR^{k \times k}$ is upper-triangular
% \STATE Define and output $\hat{\bA} := \bQ(\bQ^T\bA)$
% \end{algorithmic}

\begin{algorithm}
\SetAlgoLined
\KwIn{Matrix $\bA \in \bbR^{n \times d}$, $k \le \rk{\bA}$}
\KwOut{Approximation $\hat{\bA} \in \bbR^{n \times d}$ with $\rk{\hat{\bA}} \le k$}
\begin{itemize}

\item Sample $\bOmega \in \bbR^{d \times k}$ i.i.d. from $\mathcal{N}(0,1)$. 

\item Define $\bY := \bA\bOmega$.

\item Take QR-decomposition $\bY = \bQ\bR$, where $\bQ \in \bbR^{n \times k}$.

% \item Decompose $\bY = \bQ\bR$, where $\bQ \in \bbR^{n \times k}$ satisfies $\bQ^T\bQ = \bI_k$ and $\bR \in \bbR^{k \times k}$ is upper-triangular.

\item Define and output $\hat{\bA} := \bQ(\bQ^T\bA)$.

\end{itemize}

\caption{RSVD}

\label{alg:randomized_approx}
\end{algorithm}

The following generalization of Algorithm \ref{alg:randomized_approx}, that applies a polynomial filter to $\bA$ first, is usually used in practice: 

\begin{algorithm}[h]
\SetAlgoLined
\KwIn{Matrix $\mathbf{A} \in \mathbb{R}^{n \times d}$, $k \le \mathrm{rank}(\mathbf{A})$, polynomial filter $\chi(\mathbf{A}) = c_0+\sum_{j=1}^{p} c_{j} \mathbf{(AA^T)}^{j-1}\bA$}
\KwOut{Approximation $\hat{\mathbf{A}} \in \mathbb{R}^{n \times d}$ with $\mathrm{rank}(\hat{\mathbf{A}}) \le k$}
\begin{itemize}
\item Sample $\boldsymbol{\Omega} \in \mathbb{R}^{d \times k}$ i.i.d. from $\mathcal{N}(0,1)$. 
\item \textbf{Polynomial filtering}: Compute $\mathbf{Y} := \chi(\mathbf{A})\boldsymbol{\Omega}$
\begin{itemize}
    \item Initialize $\mathbf{Y} := c_0\boldsymbol{\Omega}$
    \item Initialize $\mathbf{Z} := \boldsymbol{\Omega}$ \quad \textit{// Stores $(\bA{\bA^T})^{j-1}\bA\boldsymbol{\Omega}$}
    \item For $j=1$, find $\bA\bOmega$
    \item For $j = 2$ to $p$:
    \begin{itemize}
        \item $\bZ' := \bA^T\mathbf{Z}$ 
        \item $\bZ := \bA\bZ'$
        \quad \textit{// Now $\mathbf{Z} = \mathbf{(AA^T)}^{j-1}\bA\boldsymbol{\Omega}$}
        \item $\mathbf{Y} := \mathbf{Y} + c_j\mathbf{Z}$ 
    \end{itemize}
\end{itemize}
% \item Decompose $\mathbf{Y} = \mathbf{Q}\mathbf{R}$, where $\mathbf{Q} \in \mathbb{R}^{n \times k}$ satisfies $\mathbf{Q}^T\mathbf{Q} = \mathbf{I}_k$ and $\mathbf{R} \in \mathbb{R}^{k \times k}$ is upper-triangular.
\item Take QR-decomposition $\bY = \bQ\bR$, where $\bQ \in \bbR^{n \times k}$
\item Define and output $\hat{\mathbf{A}} := \mathbf{Q}(\mathbf{Q}^T\mathbf{A})$.
\end{itemize}

\caption{RSVD with a polynomial filter}
\label{alg:randomized_approx_filters}
\end{algorithm}

% \textbf{Key insight}: Although the polynomial filter mathematically involves powers $(\mathbf{AA^T})^j\bA$ for $j$ up to $p$, we never explicitly compute these matrix powers. Instead, we iteratively compute $(\mathbf{AA^T})^j\bA\bOmega$ through repeated matrix-vector multiplications.

The core intuition behind Algorithm \ref{alg:randomized_approx} is that many important real-world matrices are "approximately low-rank". To approximate a matrix $\bA$, which is "approximately rank $r$", one usually takes $k > r$ in Algorithm \ref{alg:randomized_approx} to build in tolerance for the algorithm's inherent approximation error. The rank $r$ is usually referred to as {\it target rank} in the literature. To avoid confusion, we would like to highlight that the actual rank of the approximation returned by Algorithm \ref{alg:randomized_approx} is $k$ rather than $r$.

As for Algorithm \ref{alg:randomized_approx_filters}, note that $\chi(\bA) = \bA$ recovers Algorithm \ref{alg:randomized_approx}. Further intuition is that the polynomial $\chi$ acts as a spectral filter, amplifying dominant singular values while suppressing smaller ones. The choice of $\chi$ can be tailored to the spectral properties of $\mathbf{A}$, but in this work we mainly concentrate on $\chi(\bA)= (\mathbf{A}\mathbf{A}^T)^q\mathbf{A}$ after deriving the main theorem.

\subsection{Related Works}

As mentioned in the abstract, the main advantage of the present paper is that its results are {\bf tight} in the asymptotic regime. The purpose of this section is to elucidate the significance of our results compared to the existing results in the literature.

The seminal work of \cite{halko2011finding} established approximation error bounds for various randomized low-rank matrix approximations with respect to different matrix norms. More recently, \cite{tropp2023randomized} provided simplified analyses yielding cleaner bounds, that can be stated as follows in the case of the Frobenius norm:

\begin{theorem}[Theorem 8.1 in \cite{tropp2023randomized}]\label{thm: halko_bound}
    
Denote the output of Algorithm \ref{alg:randomized_approx} for input $\bA$ and Gaussian $\bOmega \in \bbR^{d \times k}$ by $\hat{\bA}(\bOmega)$  and consider an arbitrary $r < k - 1$. It holds that:
\begin{align*}
\bbE_{\bOmega \in \bbR^{d \times k}} \|\bA - \hat{\bA}(\bOmega)\|_F^2\leq \left(1 + \frac{r}{k-r-1}\right) \sum_{j>r} \sigma_j^2
\end{align*}

\end{theorem}

The work of \cite{witten2015randomized} approached the problem from a different angle, establishing {\bf worst-case} lower bounds. We do not quote the exact statement from \cite{witten2015randomized} to avoid confusion, as the approximation error in \cite{witten2015randomized} is measured with respect to the {\bf operator norm}. However, we believe that their results can be transferred to the case of the Frobenius norm with minor adjustments.  To give a flavor of their main result, they constructed examples of matrices $\bA$, for which the bounds from \cite{halko2011finding}  match the actual performance up to a multiplicative factor of $1 - \epsilon$ for any $\epsilon > 0$. This result demonstrates that the upper bounds of \cite{halko2011finding} cannot be improved for all $\bA$ simultaneously. 

Our results differ from these prior works in several ways:

\begin{enumerate}
\item {\bf Asymptotic exactness vs.\ upper bounds}: While previous works provide upper bounds (and matching lower bounds in worst case), we establish asymptotic equivalences. 

\item {\bf Full spectral characterization}: The bound in Theorem \ref{thm: halko_bound} depends only on the tail sum $\sum_{j>r} \sigma_j^2$. In contrast, our implicit characterization involves all singular values $\sigma_1, \ldots, \sigma_m$, capturing how the entire spectral distribution affects the approximation error. This gives a more nuanced picture: two matrices with the same tail sum can have different approximation errors depending on their full spectra.

\item {\bf Unified framework for filtered variants}: While previous analyses treat different algorithmic variants separately, Theorem~\ref{thm: err_filt} provides a unified framework. By introducing polynomial filters $\chi$, we can analyze many algorithms simultaneously.

\item {\bf Instance-specific vs.\ worst-case analysis}: The lower bounds of \cite{witten2015randomized} are worst-case results.  Our results are instance-specific, providing the exact asymptotic error for any given singular value distribution. 
\end{enumerate}

% In summary, while \cite{halko2011finding} and \cite{tropp2023randomized} provide practical upper bounds and \cite{witten2015randomized} establishes their worst-case tightness, our work provides exact asymptotic characterizations that depend on the full spectral structure. These results are complementary: the prior work tells us what we can guarantee in general, while our work tells us what actually happens for specific problem instances in the large-scale limit.

% As we mentioned in the abstract, the main advantage of the present paper is that its results are {\bf tight} in the asymptotic regime. The purpose of this section is elucidating the significance of our results compared to the results existing in the literature. 

% \cite{witten2015randomized}

\section{Main Results}

We would like to derive asymptotic predictions for the performance of the RSVD. This requires the following definitions: 

% \begin{definition}[Matrix Ensemble]
% A matrix ensemble is a sequence $\{\bA_i\}_{i \in \mathbb{N}}$ of random matrices where $\bA_i \in \mathbb{R}^{n_i \times d_i}$ is distributed according to a probability measure $\calP_i$ on the space of $n_i \times d_i$ real matrices. The dimensions $n_i$ and $d_i$ grow with $i$. We will denote the singular values of $\bA_i$ by $\sigma_1^{{i}}, \dots, \sigma_{m_i}^{{i}}$, where $m_i = \min(n_i,d_i)$. 
% \end{definition}

\begin{definition}[RSVD Admissible Sequence]
    We call an {\it RSVD admissible sequence} any sequence of the form $\{(n_i, d_i, m_i, k_i, \bA^{(i)})\}_{i \in \mathbb{N}}$, where $\bA^{(i)} \in \bbR^{n_i \times d_i}$ and $k_i \le \rk{\bA^{(i)}} = m_i$ are such that $k_i \to \infty$. We will denote the non-zero singular values of $\bA^{(i)}$ by $\sigma_1^{{i}}, \dots, \sigma_{m_i}^{{i}}$. 
\end{definition}

\begin{definition}[Asymptotic Equivalence]
    Given two sequences $\{a_i\}_{i \in \mathbb{N}}$ and $\{b_i\}_{i \in \mathbb{N}}$, we say that they are asymptotically equivalent if $\lim_{i \to \infty} \frac{a_i}{b_i} = 1$. We denote this by $a_i \sim b_i$. 
\end{definition}

\noindent For notational brevity, we shall henceforth drop the index $i$ and write $\bA \in \mathbb{R}^{n \times d}$ to denote a generic element of the ensemble, whose singular values are $\sigma_1, \dots, \sigma_m$, with the understanding that all statements refer to the asymptotic regime as the matrix dimensions grow. We will drop the index $i$ from the dimensions too.

% \begin{theorem}
%     Let $(n, d, k, \bA)$ be an RSVD admissible sequence. Define $\tTheta_0$ via 
%     $$\sum_{j=1}^{m} \frac{\sigma^2_j}{\tTheta_0+k\sigma^2_j} = 1$$
%     Assume that the following asymptotic equivalence holds for a constant $c$ independent of $i$: $$\tTheta_0 \sim c\|\bA\|_F^2$$  
%     Then the approximation error of the RSVD is asymptotically equivalent to the same:
%     \begin{align*}
%         \bbE_{\bOmega \in \bbR^{d \times k}} \min_{\bB \in \bbR^{k \times d}} \|\bA - \bA\bOmega \bB\|_F^2  \sim c\|\bA\|_F^2
%     \end{align*}
% \end{theorem}
\subsection{Main Theorems}
In this section we characterize precisely the approximation error of RSVD with and without polynomial filters.
\begin{theorem}[Asymptotic Performance of RSVD]\label{thm: err}
    Let $(n, d, m, k, \bA)$ be an RSVD admissible sequence. Define $\tTheta$ as the unique positive solution of
    $$\sum_{j=1}^{m} \frac{\sigma^2_j}{\tTheta+k\sigma^2_j} = 1$$
    Denote the output of Algorithm \ref{alg:randomized_approx} with a Gaussian matrix $\bOmega$ applied to $\bA$ by $\hat{\bA}(\bOmega)$. Then the following asymptotic equivalence holds: 
    \begin{align*}
        \bbE_{\bOmega \in \bbR^{d \times k}} \|\bA - \hat{\bA}(\bOmega)\|_F^2  \sim \tTheta
    \end{align*}
\end{theorem}
We require the following formal definition of the polynomial filters.
\begin{definition}[Polynomial filters]
    For the purposes of the present work, a filter is an arbitrary polynomial $\chi: \bbR \to \bbR$. Given a matrix $\bX$ with SVD decomposition $\bX = \bU \bSigma \bV$, we define $\chi(\bX) = \bU \chi(\bSigma) \bV$, where $\chi(\bSigma)$ is applied element wise to each element of the diagonal. 
\end{definition}

\begin{theorem}\label{thm: err_filt}
    Let $(n, d, m, k, \bA)$ be an RSVD admissible sequence and $\chi$ be a filter. Define $\theta_0$ and $\tTheta$ as the unique positive numbers satisfying: 
    \begin{align*}
        \sum_{j=1}^m \frac{1}{1 + \theta_0 \chi(\sigma_j)^2} &= m - k \\
        \sum_{j=1}^m \frac{\sigma_j^2}{1 + \theta_0\chi(\sigma_j)^2} &= \tTheta 
    \end{align*}
     Denote the output of Algorithm \ref{alg:randomized_approx_filters} with a Gaussian matrix $\bOmega$ applied to $\bA$ by $\hat{\bA}(\bOmega)$. Then the following asymptotic equivalence holds :
    \begin{align*}
        \bbE_{\bOmega \in \bbR^{d \times k}}  \|\bA - \hat{\bA}(\bOmega)\|_F^2  \sim \tTheta
    \end{align*}
\end{theorem}

\begin{remark}

It should be noted that the expression we get for the approximation error for $\chi(x) = x$ in Theorem \ref{thm: err_filt} is not the same as the expression resulting from Theorem \ref{thm: err}. It can actually be shown that the resulting expressions are equivalent, but we prefer to keep the current form of Theorem \ref{thm: err}, as the corresponding equation appears simpler.   
    
\end{remark}

\subsection{Implications}

Define 
\begin{align*}
    & C = C(k,r,m) : = \frac{(m-k)k}{(m-r)(k-r)} 
    % & c = c(k,r,d) : = \frac{d-k}{d-r}
\end{align*} 

The following propositions give asymptotic upper bounds for the approximation error of the RSVD in terms of an arbitrary $r < k$. Note that these bounds are marginally better than Theorem \ref{thm: halko_bound} since $r < k$. 
%known as the {\it target rank} in the literature:

\begin{proposition}\label{prop: upp_bnd}
In terms of the notation from Theorem \ref{thm: err}, the following holds:
\begin{align*}
   \sum_{j = k + 1}^m \sigma_j^2 \le  \tTheta \le C \sum_{j = r + 1}^m \sigma_j^2
\end{align*}
\end{proposition}
Note that the lower bound in Proposition \ref{prop: upp_bnd} follows readily from Eckart-Young theorem. The following result provides an asymptotic upper bound for the approximation error of arbitrary polynomial filters.
\begin{proposition}\label{prop: upp_bnd_filt}
    In terms of the notation from Theorem \ref{thm: err_filt}, the following holds:
    \begin{align*}
        \tTheta \le 
 \sum_{j = r + 1}^m  \sigma_j^2 + \frac{C}{k}\sum_{j=1}^r \frac{\sigma_j^2}{\chi(\sigma_j)^2} \cdot \sum_{j = r + 1}^m\chi(\sigma_j)^2 
    \end{align*}
\end{proposition}
It turns out, in the case of the Bilevel Ensemble and the Power law Ensemble, the asymptotic error and the structure of an optimal filter can characterized explicitly as follows.
\begin{corollary}\label{cor: ensemble}
\begin{itemize}
    \item {\bf(Bilevel Ensemble).} Assume that \\$\sigma_1 = \dots = \sigma_r = a$ and $\sigma_{r+1} = \dots = \sigma_{m} = b$, with $a > b$. Then the approximation error would be
    \begin{align}\label{eq: cor_bilevel}
        &\bbE_{\bOmega \in \bbR^{d \times k}}  \|\bA - \hat{\bA}(\bOmega)\|_F^2   \\& \sim \frac{1}{2} \Bigl((r-k) a^2 + (m - r- k) b^2 + \nonumber \\ & \sqrt{((r-k) a^2 + (m - r- k) b^2)^2 + 4 a^2 b ^2 (m-k)} \Bigr) \nonumber
    \end{align}
    In this setting for $k \ge r +1$, the optimal polynomial filter is of form $\chi(\cdot):=(\cdot)^q$ where $q \rightarrow \infty$ and yields an error of $(m-k) b^2$, the fundamental lower bound from Proposition \ref{prop: upp_bnd}.
    \item {\bf(Power law Ensemble).} Assume that $\sigma_i = i^{-\alpha}$, then 
    \begin{align}\label{eq: cor_power_law}
        & \bbE_{\bOmega \in \bbR^{d \times k}}  \|\bA - \hat{\bA}(\bOmega)\|_F^2  \sim \frac{1}{\Bigl(k \sinc(\frac{\pi}{2\alpha} )\Bigr)^{2\alpha}} \\
        &\text{where }\sinc(x) = \frac{\sin(x)}{x} \nonumber
    \end{align}
    In this setting, $\chi(\cdot):=(\cdot)^q$ where $q \rightarrow \infty$ yields an error of $\sum_{i=k+1}^m i^{-2\alpha}$, the fundamental lower bound from Proposition \ref{prop: upp_bnd}.
\end{itemize}
\end{corollary}
% \begin{remark}
%     For the bilevel ensemble, the upper bound in Theorem \ref{thm: halko_bound} yields $$\frac{k-1}{k-r-1} (m-r) b^2$$ whereas an straightforward upper bound on the error derived in \ref{cor: ensemble} is $$\Bigl|(r-k) a^2 + (m - r- k) b^2\Bigr| + 2 a b \sqrt{m-k}$$ Assuming $r+k < m$ and since $m$ is growing, we can drop $2 a b \sqrt{m-k}$ in comparison with $(m - r- k) b^2$. Thus our results yield an upper bound of $(m-r-k) b^2$ compared to that of $\frac{k-1}{k-r-1} (m-r) b^2$, where in the latter, taking $k=r+2$ adds an additional factor of $k$ in the upper bound.
% \end{remark}
\begin{remark}
    For the bilevel ensemble, the upper bound in Theorem \ref{thm: halko_bound} yields 
    \begin{align}\label{eq: halko_bilevel}
    \frac{k-1}{k-r-1} (m-r) b^2
    \end{align}
    whereas, assuming $k = o(m)$ and $\frac{a}{b} = \Theta(1)$, \eqref{eq: cor_bilevel} asymptotically equals  $$(m - k) b^2$$ If the difference $k-r$ is constant, we conclude that our prediction \eqref{eq: cor_bilevel} is tighter than \eqref{eq: halko_bilevel} by a factor of $\Theta(k)$. 
\end{remark}
\begin{remark}
    For the power law ensemble, the upper bound in Theorem \ref{thm: halko_bound} yields approximately
    \begin{align}\label{eq: halko_power_law}
         \frac{k-1}{k-r-1} \cdot \frac{1}{2 \alpha - 1} \cdot \frac{1}{r^{2 \alpha-1}}
    \end{align}
    The bound \eqref{eq: halko_power_law} is $\Theta(k^2)$- fold greater than \eqref{eq: cor_power_law} if the oversampling $k-r$ is constant and $\Theta(k)$- fold greater if $r$ is a finite fraction of $k$.
\end{remark}

\section{Our Approach}

\subsection{Variational reformulation}
The following key lemma provides a variation formulation for the RSVD algorithm which allows us to apply Gaussian Comparison Inequalities\cite{Gordon1985SomeIF, Gordon88}. 
\begin{lemma}\label{lm: the_obj}
Let $\chi$ be a polynomial filter and assume that $\mathrm{rank}(\chi(\mathbf{A})\boldsymbol{\Omega}) = k$ under the notation of Algorithm \ref{alg:randomized_approx_filters}. Then the output of Algorithm \ref{alg:randomized_approx_filters} can be found as $\hat{\mathbf{A}} = \chi(\mathbf{A})\boldsymbol{\Omega} \mathbf{B}$, where $\mathbf{B}$ is defined uniquely via the following objective:
\begin{align*}
   \min_{\mathbf{B} \in \mathbb{R}^{k \times d}} \|\mathbf{A} - \chi(\mathbf{A})\boldsymbol{\Omega} \mathbf{B}\|_F^2
\end{align*}
When $\chi(x) = x$, this description captures the output of Algorithm \ref{alg:randomized_approx}. 
\end{lemma}

\begin{remark}
Define the approximation error of Algorithm \ref{alg:randomized_approx_filters} as
\begin{align}\label{eq: appr_err}
    \mathbb{E}_{\boldsymbol{\Omega} \in \mathbb{R}^{d \times k}} \|\mathbf{A} - \hat{\mathbf{A}}\|_F^2
\end{align}
We see from Lemma \ref{lm: the_obj} that \eqref{eq: appr_err} equals
\begin{align*}
    \mathbb{E}_{\boldsymbol{\Omega} \in \mathbb{R}^{d \times k}} \min_{\mathbf{B} \in \mathbb{R}^{k \times d}} \|\mathbf{A} - \chi(\mathbf{A})\boldsymbol{\Omega} \mathbf{B}\|_F^2
\end{align*}
\end{remark}

\subsection{Proof of Theorem \ref{thm: err} }

We present the proof of Theorem \ref{thm: err} below. The proof of Theorem \ref{thm: err_filt} follows the same steps and is omitted here due to the lack of space. 

\begin{proof}[Proof of Theorem \ref{thm: err} ]
Consider the truncated SVD decomposition of $\mathbf{A}$:
\begin{align*}
   & \mathbf{A} = \mathbf{U} \boldsymbol{\Sigma} \mathbf{V}, \text{ where } m = \rk\bA,
    \; \boldsymbol{\Sigma} = \sum_{1 \le i \le m} \sigma_i \mathbf{e}_i\mathbf{e}_i^T, \nonumber\\
    &\mathbf{U} \in \mathbb{R}^{n \times m}, \mathbf{V} \in \mathbb{R}^{m \times d}, \quad
    \mathbf{U}^T\mathbf{U} = \mathbf{I}_m = \mathbf{V}\mathbf{V}^T \nonumber
\end{align*}

Due to rotational invariance of the Frobenius norm we have: 
\begin{align*}
   \|\mathbf{A} - \mathbf{A}\boldsymbol{\Omega} \mathbf{B}\|_F^2  &=  \|\mathbf{U}\boldsymbol{\Sigma}\mathbf{V} -\mathbf{U}\boldsymbol{\Sigma}\mathbf{V}\boldsymbol{\Omega} \mathbf{B}\|_F^2 \\
     & = \|\boldsymbol{\Sigma} -\boldsymbol{\Sigma}(\mathbf{V}\boldsymbol{\Omega})(\mathbf{B}\mathbf{V}^T)\|_F^2 \nonumber
\end{align*}

Noting that $\mathbf{V}\boldsymbol{\Omega} \in\mathbb{R}^{m \times k}$ is also i.i.d. standard Gaussian and $\mathrm{rank}(\mathbf{V})=m$, we redefine $\boldsymbol{\Omega} := \mathbf{V}\boldsymbol{\Omega} \in \mathbb{R}^{m \times k}$, $\mathbf{B} := \mathbf{B}\mathbf{V}^T$ and arrive at the following: 
\begin{align*}
    \mathbb{E}_{\boldsymbol{\Omega} \in \mathbb{R}^{m \times k}} \min_{\mathbf{B} \in \mathbb{R}^{k \times m}} \|\boldsymbol{\Sigma} - \boldsymbol{\Sigma}\boldsymbol{\Omega} \mathbf{B}\|_F^2
\end{align*}

Introducing the Fenchel dual variable $\mathbf{C} \in \mathbb{R}^{m \times m}$ for the quadratic loss $\|\cdot\|_F^2$ we obtain:
\begin{align}\label{eq: fen_dual}
    &\mathbb{E}_{\boldsymbol{\Omega} \in \mathbb{R}^{m \times k}} \\ &\min_{\mathbf{B} \in \mathbb{R}^{k \times m}} \max_{\mathbf{C} \in \mathbb{R}^{m \times m}} \mathrm{tr}\left(\mathbf{C}^T\left(\boldsymbol{\Sigma} - \boldsymbol{\Sigma}\boldsymbol{\Omega} \mathbf{B}\right)\right) - \frac{\|\mathbf{C}\|_F^2}{4} \nonumber
\end{align}

Denote the optimization variables
\begin{align*}
    \mathbf{C} &= (\mathbf{C}_1, \dots, \mathbf{C}_m), \quad \mathbf{C}_i \in \mathbb{R}^m\\ \mathbf{B} &= (\mathbf{B}_1, \dots, \mathbf{B}_m), \quad  \mathbf{B}_i \in \mathbb{R}^k
\end{align*}

Expression \eqref{eq: fen_dual} decomposes as:
\begin{align*}
    &\sum_{i=1}^{m} \mathbb{E}_{\boldsymbol{\Omega} \in \mathbb{R}^{m \times k}} \min_{\mathbf{B}_i \in \calS_{\bB}} \max_{\mathbf{C}_i \in \mathbb{R}^{m}} \mathbf{C}_i^T\left(\sigma_i\mathbf{e}_i - \boldsymbol{\Sigma}\boldsymbol{\Omega} \mathbf{B}_i\right) - \frac{\|\mathbf{C}_i\|^2}{4} \nonumber
\end{align*}

Define the constraint sets
\begin{align*}
    & \calS_{\bB} := \Bigl\{\bB_i: \|\bB_i\| \le \frac{2}{\sqrt{d} - \sqrt{k}}\frac{\sigma_1(\bA)}{\sigma_m(\bA)} \Bigr\} \\
    & \calS_{\bC} := \{\bC_i: \|\bC_i\| \le \sigma_i \}
\end{align*}   
Consider the following pair of optimizations:
\begin{align*}
    & \Phi_i(\boldsymbol{\Omega}) :=  \max_{\mathbf{C}_i \in \calS_{\bC}} \min_{\mathbf{B}_i \in \calS_{\bB}} \mathbf{C}_i^T \boldsymbol{\Omega} \mathbf{B}_i + \sigma_i\mathbf{C}_{ii}- \frac{\|\mathbf{C}_i\|^2}{4}
\end{align*}
\begin{align}
    \phi_i(\mathbf{g}, \mathbf{h}) :=  \max_{\mathbf{C}_i \in \calS_{\bC}} \min_{\mathbf{B}_i \in \calS_{\bB}} & \mathbf{C}_i^T \boldsymbol{\Sigma} \mathbf{g} \|\mathbf{B}_i\|   + \mathbf{B}_i^T \mathbf{h}\|\boldsymbol{\Sigma}\mathbf{C}_i\| \nonumber \\
    & + \sigma_i\mathbf{C}_{ii}- \frac{\|\mathbf{C}_i\|^2}{4} \label{eq: PO_AO}
\end{align}
We further define $\phi^{HW}_i(\mathbf{h}) := \frac{\sigma_i^2}{1 +\theta_{\bh}  \sigma_i^2}$ and $ c_i := \frac{\sigma_i^2}{1 +\theta  \sigma_i^2}$ for $i \in [m]$ where $\theta_\bh$ and $\theta$ are defined via
\begin{align*}
    \sum_{j=1}^m \frac{1}{1 + \theta_{\bh} \sigma_j^2} &= m - \|\mathbf{h}\|^2 \\
    \sum_{j=1}^m \frac{1}{1 + \theta \sigma_j^2} &= m - k
\end{align*}

Going back to \eqref{eq: fen_dual}, we observe that it can be written as
\begin{align*}
& \mathbb{E}_{\boldsymbol{\Omega} \in \mathbb{R}^{m \times k}} \min_{\mathbf{B} \in \mathbb{R}^{k \times m}} \|\boldsymbol{\Sigma} - \boldsymbol{\Sigma}\boldsymbol{\Omega} \mathbf{B}\|_F^2 = \\
& \sum_{i=1}^m \bbE \max_{\mathbf{C}_i \in \mathbb{R}^m} \min_{\mathbf{B}_i \in \mathbb{R}^k} \mathbf{C}_i^T \boldsymbol{\Omega} \mathbf{B}_i + \sigma_i\mathbf{C}_{ii}- \frac{\|\mathbf{C}_i\|^2}{4}
\end{align*}

From Lemma \ref{lm: solution_norm_bounds},
 \begin{align*}
 & \bbE_{\boldsymbol{\Omega} \in \mathbb{R}^{m \times k}}  \max_{\mathbf{C}_i \in \mathbb{R}^m} \min_{\mathbf{C}_i \in \mathbb{R}^k} \mathbf{C}_i^T \boldsymbol{\Omega} \mathbf{B}_i + \sigma_i\mathbf{C}_{ii}- \frac{\|\mathbf{C}_i\|^2}{4} \\
 & \sim \mathbb{E}_{\boldsymbol{\Omega} \in \mathbb{R}^{m \times k}} \Phi_i(\boldsymbol{\Omega}) 
 \end{align*}
Indeed, according to Lemma \ref{lm: solution_norm_bounds}, the optimal $\bC_i$ and $\bB_i$ belong in the sets $\calS_{\bC}$ and $\calS_{\bB}$ with probability approaching $1$ as $m \to \infty$ and the objective is bounded outside of these sets according to the same lemma. In Lemma \ref{lemma: po_ao_comparison} located in the Appendix, we prove that 
\begin{align*}
\left|\mathbb{E}_{\boldsymbol{\Omega} \in \mathbb{R}^{m \times k}} \Phi_i(\boldsymbol{\Omega}) - c_i\right| \le 2\left| \mathbb{E} \phi_i(\mathbf{g},\mathbf{h}) - c_i\right|
\end{align*}

We then proceed to apply the triangle inequality:
\begin{align*}
2\left|\mathbb{E} \phi_i(\mathbf{g},\mathbf{h}) - c_i\right|  \le & 2 |\mathbb{E} \phi_i(\mathbf{g},\mathbf{h}) - \mathbb{E}\phi^{HW}_i(\mathbf{h})| \nonumber\\
&\quad + 2\left| \mathbb{E} \phi^{HW}_i(\mathbf{h}) - c_i\right| \nonumber
\end{align*}

After that, we prove that $\mathbb{E} \phi_i(\mathbf{g},\mathbf{h}) \sim \mathbb{E} \phi^{HW}_i(\mathbf{h})$ and $ \mathbb{E} \phi^{HW}_i(\mathbf{h}) \sim c_i$ in the Appendix, in Lemma \ref{lm: HW_bound} and Lemma \ref{lm: phi_hw_exp} respectively. Hence, it suffices to prove that $ \sum_{i=1}^m c_i = \tTheta$. The latter is a straightforward algebra verification:
\begin{align*}
    & \sum_{i=1}^m c_i = \frac{\sigma_i^2}{1 +\theta  \sigma_i^2} = \frac{1}{\theta}\sum_{i=1}^m \left(1 - \frac{1}{1 +\theta  \sigma_i^2}\right) = \frac{k}{\theta}
\end{align*}

It only remains to check that $\tTheta = \frac{k}{\theta}$, which is true since, by definition of $\theta$,
\begin{align*}
    \sum_{i=1}^m \frac{\sigma_i^2}{\frac{k}{\theta} + k\sigma_i^2 } = \frac{1}{k} \sum_{i=1}^m \frac{\theta\sigma_i^2}{1 + \theta\sigma_i^2} =\frac{1}{k} \sum_{i=1}^m \left(1 - \frac{1}{1 + \theta\sigma_i^2}\right) = 1 
\end{align*}

\end{proof}

\section{Numerical simulations}

We sample matrices with power-law and bilevel spectra and plot the error of the RSVD (Algorithm \ref{alg:randomized_approx}) and the RSVD with a polynomial filter corresponding to the power iteration method (Algorithm \ref{alg:randomized_approx_filters}) with $\bA \mapsto (\bA\bA^T)^q \bA$, against the value of $k$ along with the corresponding theoretical prediction provided in Corollary \ref{cor: ensemble}. As a baseline, we also plotted the fundamental lower bound and also the upper bound from \ref{prop: upp_bnd} to compare the performance of various polynomial filters and the best achievable error:
\begin{align*}
   \sum_{j = k + 1}^m \sigma_j^2 \le  \tTheta \le C \sum_{j = r + 1}^m \sigma_j^2
\end{align*}
The results and the details of the ensembles are presented in Figures \ref{subfig:bilvl}, \ref{subfig:pwrlaw}. We observe a close match between the theoretical predictions and the empirical errors for large enough values of $k$. Moreover, even though, technically, the results hold as $k \to \infty$, we observe that $k \approx 15$ already suffices for $n,d \approx 1000$ and even the gap for the values of $k < 15$ appears to be rather modest. Furthermore, the comparison with the upper bound provided by Proposition \ref{prop: upp_bnd}, a tighter bound compared to that of \cite{halko2011finding, tropp2023randomized} demonstrates the gap between the precise values and the upper bound. In fact for the bilevel case, the values for the upper bound are as follows: [2919.42, 1091.47, 829.08, 723.24,  665.57, 628.98, 603.48, 584.53], which are substantially larger than the values provided in Fig. \ref{subfig:bilvl}. Additionally, we observe that as predicted by Corollary \ref{cor: ensemble}, increasing the degree of the polynomial suffices to improve the approximation error and for the power law ensemble, remarkably setting $q=4$ provides an error sufficiently close to the fundamental lower bound, while for the bilevel ensemble, the gain in performance is modest.

% \begin{figure}[h]
% \centering
% \begin{subfigure}
% \centering
% \includegraphics[width=\linewidth]{rsvd_bilevel.png}
%     \caption{Bilevel with $r = 50$}
% \end{subfigure}
% % \begin{subfigure}
% % \centering\includegraphics[width=\linewidth]{rsvd_pwr_law.png}
% %     \caption{Power law, i.e. $\sigma_i = i^{-s}$, $s=1$. Fixed $r=3$.}
% % \end{subfigure}
% \end{figure}

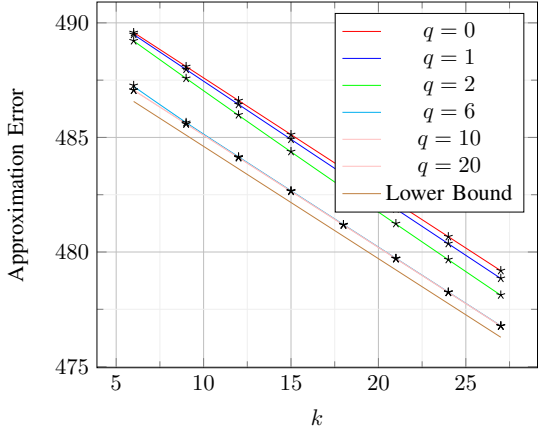
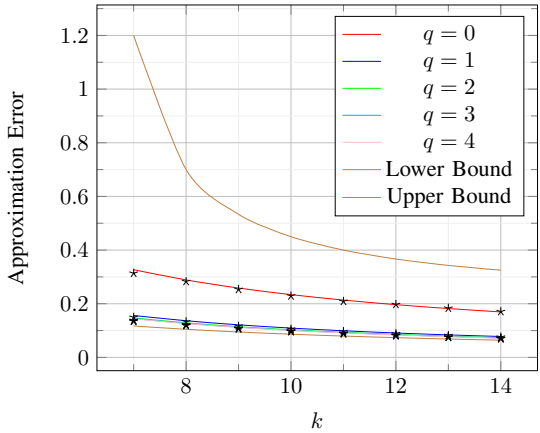
\begin{figure}[!h] 
    \centering
    \subfloat [Bilevel with $n=d=1000$, $r=50$, $a=1$, $b=0.7$] {
    \label{subfig:bilvl}
    \resizebox{0.4\textwidth}{!}{%
        \begin{tikzpicture}
        \begin{axis}[
          xlabel={$k$},
          ylabel=Approximation Error,
          legend pos = north east,
          grid = both,
            minor tick num = 1,
            major grid style = {lightgray},
            minor grid style = {lightgray!25},
          ]
        \addplot[mark = star, black, only marks,forget plot] 
        table[ y=q0, x=k]{bilvl_r=50_R=100_n=d=1000_a=1_b=07.dat};
        % \addlegendentry{}
        \addplot[mark = star, black, only marks, forget plot] 
        table[ y=q1, x=k]{bilvl_r=50_R=100_n=d=1000_a=1_b=07.dat};
        % \addlegendentry{$B$}
        \addplot[mark = star, black, only marks, forget plot] 
        table[ y=q2, x=k]{bilvl_r=50_R=100_n=d=1000_a=1_b=07.dat};
        % \addlegendentry{$\chi^2$}
        \addplot[mark = star, black, only marks, forget plot] 
        table[ y=q3, x=k]{bilvl_r=50_R=100_n=d=1000_a=1_b=07.dat};
        % \addlegendentry{$e_p$}
        \addplot[mark = star, black, only marks, forget plot] 
        table[ y=q4, x=k]{bilvl_r=50_R=100_n=d=1000_a=1_b=07.dat};
        \addplot[mark = star, black, only marks, forget plot] 
        table[ y=q5, x=k]{bilvl_r=50_R=100_n=d=1000_a=1_b=07.dat};
        
        \addplot[smooth, thin, red] 
        table[ y=qpred0, x=k]{bilvl_r=50_R=100_n=d=1000_a=1_b=07.dat};
        \addlegendentry{$q=0$}
        \addplot[smooth, thin, blue] 
        table[ y=qpred1, x=k]{bilvl_r=50_R=100_n=d=1000_a=1_b=07.dat};
        \addlegendentry{$q=1$}
        \addplot[smooth, thin, green] 
        table[ y=qpred2, x=k]{bilvl_r=50_R=100_n=d=1000_a=1_b=07.dat};
        \addlegendentry{$q=2$}
        \addplot[smooth, thin, cyan] 
        table[ y=qpred3, x=k]{bilvl_r=50_R=100_n=d=1000_a=1_b=07.dat};
        \addlegendentry{$q=6$}
        \addplot[smooth, thin, pink] 
        table[ y=qpred4, x=k]{bilvl_r=50_R=100_n=d=1000_a=1_b=07.dat};
        \addlegendentry{$q=10$}
        \addplot[smooth, thin, pink] 
        table[ y=qpred5, x=k]{bilvl_r=50_R=100_n=d=1000_a=1_b=07.dat};
        \addlegendentry{$q=20$}
        \addplot[smooth, thin, brown] 
        table[ y=lwbnd, x=k]{bilvl_r=50_R=100_n=d=1000_a=1_b=07.dat};
        \addlegendentry{Lower Bound}
        % \addplot[smooth, thin, brown] 
        % table[ y=upbnd, x=k]{bilvl_r=50_R=100_n=d=1000_a=1_b=07.dat};
        % \addlegendentry{Upper Bound}
        \end{axis}
    \end{tikzpicture}   
        } 
    }\\
   \subfloat [Power law with $\alpha = 1$, $d=1000$] {
    \label{subfig:pwrlaw}
    \resizebox{0.4\textwidth}{!}{%
        \begin{tikzpicture}
        \begin{axis}[
          xlabel={$k$},
          ylabel=Approximation Error,
          legend pos = north east,
          grid = both,
            minor tick num = 1,
            major grid style = {lightgray},
            minor grid style = {lightgray!25},
          ]
        \addplot[mark = star, black, only marks,forget plot] 
        table[ y=q0, x=k]{pwrlaw_alpha=1_d=1000_r=5_R=15_a=1_b=07.dat};
        % \addlegendentry{}
        \addplot[mark = star, black, only marks, forget plot] 
        table[ y=q1, x=k]{pwrlaw_alpha=1_d=1000_r=5_R=15_a=1_b=07.dat};
        % \addlegendentry{$B$}
        \addplot[mark = star, black, only marks, forget plot] 
        table[ y=q2, x=k]{pwrlaw_alpha=1_d=1000_r=5_R=15_a=1_b=07.dat};
        % \addlegendentry{$\chi^2$}
        \addplot[mark = star, black, only marks, forget plot] 
        table[ y=q3, x=k]{pwrlaw_alpha=1_d=1000_r=5_R=15_a=1_b=07.dat};
        % \addlegendentry{$e_p$}
        \addplot[mark = star, black, only marks, forget plot] 
        table[ y=q4, x=k]{pwrlaw_alpha=1_d=1000_r=5_R=15_a=1_b=07.dat};
        
        \addplot[smooth, thin, red] 
        table[ y=qpred0, x=k]{pwrlaw_alpha=1_d=1000_r=5_R=15_a=1_b=07.dat};
        \addlegendentry{$q=0$}
        \addplot[smooth, thin, blue] 
        table[ y=qpred1, x=k]{pwrlaw_alpha=1_d=1000_r=5_R=15_a=1_b=07.dat};
        \addlegendentry{$q=1$}
        \addplot[smooth, thin, green] 
        table[ y=qpred2, x=k]{pwrlaw_alpha=1_d=1000_r=5_R=15_a=1_b=07.dat};
        \addlegendentry{$q=2$}
        \addplot[smooth, thin, cyan] 
        table[ y=qpred3, x=k]{pwrlaw_alpha=1_d=1000_r=5_R=15_a=1_b=07.dat};
        \addlegendentry{$q=3$}
        \addplot[smooth, thin, pink] 
        table[ y=qpred4, x=k]{pwrlaw_alpha=1_d=1000_r=5_R=15_a=1_b=07.dat};
        \addlegendentry{$q=4$}
        \addplot[smooth, thin, brown] 
        table[ y=lwbnd, x=k]{pwrlaw_alpha=1_d=1000_r=5_R=15_a=1_b=07.dat};
        \addlegendentry{Lower Bound}
        \addplot[smooth, thin, brown] 
        table[ y=upbnd, x=k]{pwrlaw_alpha=1_d=1000_r=5_R=15_a=1_b=07.dat};
        \addlegendentry{Upper Bound}
        \end{axis}
    \end{tikzpicture}   
        } 
    }
    \caption{Approximation Error for two ensembles of choice}

\end{figure}

\section{Conclusion and Future Work}
We have derived exact asymptotic characterizations for RSVD performance that capture the full spectral structure of the data matrix and naturally extend to polynomial-filtered variants. Our results complement existing literature by bridging the gap between worst-case guarantees and actual performance: while prior work establishes what can be guaranteed in general, our analysis reveals what actually occurs for specific problem instances in the large-scale limit.
Our approach opens several avenues for future research. The framework developed here can be extended to investigate other properties of RSVD approximation error, including finite-sample corrections to our asymptotic results and the universality of approximation error across different choices of sketching matrices $\bOmega$. These directions would further enhance our understanding of RSVD behavior beyond the asymptotic regime and across diverse randomization strategies.

% \section{Conclusion and future works}
% Our work provides exact asymptotic characterizations of the performance of the RSVD that depend on the full spectral structure and extend directly to RSVD with filters. These results are complementary to the results in the literature: the prior work tells us what we can guarantee in general, while our work tells us what actually happens for specific problem instances in the large-scale limit. We believe that our approach is of independent interest as it can pave a way for the investigation of other properties of the approximation error for RSVD. These include, for example,  extending the presented analysis to the nonasymptotic regime using and investigating universality of the reconstruction error with respect to the choice of the sketching matrix $\bOmega$. 

\section{Acknowledgments}

The authors are extremely grateful to Joel Tropp for sharing insights that lead to this paper. 

\section{Appendix}
\subsection{Technical Lemmata}
\begin{proof}[Proof of Lemma \ref{lm: the_obj}]Let $\mathbf{Y} = \chi(\mathbf{A})\boldsymbol{\Omega}$. 
The optimization problem is a least squares problem. The optimal solution is:
\begin{align*}
   \mathbf{B}^* = (\mathbf{Y}^T\mathbf{Y})^{-1}\mathbf{Y}^T\mathbf{A} = \mathbf{Y}^{\dagger}\mathbf{A} \nonumber
\end{align*}

Using the QR decomposition:
\begin{align*}
   \mathbf{Y}^{\dagger} &= (\mathbf{Q}\mathbf{R})^{\dagger} = \mathbf{R}^{-1}\mathbf{Q}^T
\end{align*}

Therefore:
\begin{align*}
   \mathbf{B}^* &= \mathbf{R}^{-1}\mathbf{Q}^T\mathbf{A}
\end{align*}

The approximation given by this optimal $\mathbf{B}^*$ is:
\begin{align*}
   & \hat{\mathbf{A}} = \mathbf{Y}\mathbf{B}^* 
   = \mathbf{Q}\mathbf{R} \mathbf{R}^{-1}\mathbf{Q}^T\mathbf{A} = \mathbf{Q}\mathbf{Q}^T\mathbf{A} \nonumber
\end{align*}

% This is exactly the output of the polynomial-filtered algorithm, which computes the projection of $\mathbf{A}$ onto the column space of $\mathbf{Y} = \chi(\mathbf{A})\boldsymbol{\Omega}$. Since $\mathrm{rank}(\mathbf{Y}) = k$, the matrix $\mathbf{Y}^T\mathbf{Y}$ is invertible, making $\mathbf{B}^*$ unique.
\end{proof}

For all subsequent lemmas we operate under the notation and assumptions of Theorem \ref{thm: err} and Theorem \ref{thm: err_filt}. 

\begin{lemma}\label{lm: solution_norm_bounds}

Let $\mathbf{B}_i$ and $\mathbf{C}_i$ be the solutions to the optimization problem:
$$\min_{\mathbf{B}_i} \|\boldsymbol{\Sigma}_i - \boldsymbol{\Sigma} \boldsymbol{\Omega} \mathbf{B}_i \|^2$$

Then, 

1. With probability at least $1 - e^{-\frac{(\sqrt{d}-\sqrt{k})^2}{8}}$, it holds that 
\begin{align}
    \|\bB^*\| \le \frac{2}{\sqrt{d} - \sqrt{k}}\kappa(\bA) 
\end{align}
where $\kappa(\bA) = \frac{\|\bA\|_2}{\sigma_{\min}(\bA)}$ is the condition number of $\bA$.

2. The following always holds:
   $$\|\mathbf{C}_i\|_2 \leq \sigma_i$$

\end{lemma}

% \begin{remark}
%     Lemma \ref{lm: solution_norm_bounds} is the reason we need to assume $\rk\bA = \min(d,n)$ in the definition of an RSVD admissible sequence. Note that we do not need to impose any lower bounds on $\kappa(\bA)$ and simply require $\sigma_{\min}(\bA) > 0$. 
% \end{remark}

\begin{proof}
1. From Lemma \ref{lm: the_obj}, we have $\bB^* = \bR^{-1}\bQ^T\bA$. Using the submultiplicativity property:
\begin{align*}
     \|\bB^*\|_{2} = \|\bR^{-1}\bQ^T\bA\|_{2} \leq \|\bR^{-1}\|_{2} \|\bQ^T\bA\|_{2} \leq \|\bR^{-1}\|_2 \|\bA\|_2 
\end{align*}
% Since $\bQ\bR = \bA\bOmega$ and $\bQ^T\bQ = \bI_k$:
Applying standard RMT results to $\sigma_{\min}(\bOmega)$ we are done since:
$$\|\bR^{-1}\|_2 = \frac{1}{\sigma_{\min}(\bR)} = \frac{1}{\sigma_{\min}(\bA\bOmega)} \le \frac{1}{\sigma_{\min}(\bA)}\frac{1}{\sigma_{\min}(\bOmega)}$$

2. Plug in the feasible solution $\mathbf{B}_i = 0$.

\end{proof}
Then we obtain:
\begin{lemma}\label{lemma: po_ao_comparison}
Let $ \gamma \sim \mathcal{N}(0,1)$ and
\begin{align*}
    \Psi_i(\bOmega, \gamma) := \max_{\mathbf{C}_i \in \calS_{\bC}} \min_{\mathbf{C}_i \in \calS_{\bB}} & \mathbf{C}_i^T \boldsymbol{\Omega} \mathbf{B}_i + \sigma_i\mathbf{C}_{ii} \\
    & - \frac{\|\mathbf{C}_i\|^2}{4} + \gamma\|\bC_i\| 
\end{align*}

The following inequalities hold for any $c, t \in \mathbb{R}$:
\begin{align*}
1. \quad \mathbb{P} \left(| \Phi_i(\bOmega) - c| > t \right) \le 2 \mathbb{P}\left(|\Psi_i(\bOmega, \gamma) - c| > t \right)        
\end{align*}
 \begin{align*} 2. \quad
\mathbb{P}\left(|\Psi_i(\bOmega, \gamma) - c| > t \right) \le \mathbb{P}\left(|\phi_i(\mathbf{g}, \mathbf{h}) - c| > t \right)        
\end{align*}
\end{lemma}

\begin{proof}
\begin{align}\label{eq: total_prob}
1. \quad \mathbb{P} \left(| \Phi_i(\bOmega) - c| > t \right) &= \mathbb{P} \left( \Phi_i(\bOmega)  > c + t \right)\nonumber \\
&\quad + \mathbb{P} \left(\Phi_i(\bOmega) < c - t \right) \nonumber \\
\mathbb{P}\left(|\Psi_i(\bOmega, \gamma) - c| > t \right) &= \mathbb{P} \left( \Psi_i(\bOmega, \gamma)  > c + t \right) \\
&\quad + \mathbb{P} \left(\Psi_i(\bOmega, \gamma) < c - t \right) \nonumber
\end{align}
Conditioning on $\gamma > 0$ and $\gamma < 0$ respectively:
\begin{align}\label{eq: conditioning}
\mathbb{P} \left( \Phi_i(\bOmega)  > c + t \right) &\leq  2 \mathbb{P}\left(\Psi_i(\bOmega, \gamma) > c + t \right) \\
\mathbb{P} \left(\Phi_i(\bOmega) < c - t \right) &\leq  2 \mathbb{P}\left(\Psi_i(\bOmega, \gamma) < c - t \right) \nonumber
\end{align}

Hence, we are done by combining equations \eqref{eq: conditioning} and \eqref{eq: total_prob}. 

2. This follows from Gordon's Min-Max theorem applied to the objective and to its convex dual. The details are identical to those in the proof of Theorem II.I in \cite{thrampoulidis2014gaussian} and the proof of Lemma 1 in \cite{akhtiamov2024novel}. 
\end{proof}

\begin{lemma}\label{lm: HW_bound}
$$\mathbb{E} \phi_i(\mathbf{g}, \mathbf{h})   \sim \bbE \phi^{HW}_i(\mathbf{h}) $$
\end{lemma}

\begin{proof}
We will start with analyzing the objective without the expectation first. Continuing from the formulation provided in \eqref{eq: PO_AO}. By convexity-concavity of the objective in $\Phi_i$, we can swap the min and max and write
\begin{align*}
    \max_{\mathbf{C}_i \in \calS_{\bC}} \min_{\mathbf{B}_i \in \calS_{\bB}}  \mathbf{C}_i^T \boldsymbol{\Sigma} \mathbf{g} \|\mathbf{B}_i\| + \mathbf{B}_i^T \mathbf{h} \|\boldsymbol{\Sigma}\mathbf{C}_i\| + \sigma_i\mathbf{C}_{ii}- \frac{\|\mathbf{C}_i\|^2}{4}
\end{align*}
We let $\bB_i = \beta \hat{\bB}_i$ where $\|\hat{\bB}_i\|_2 = 1$. We perform the optimization over the direction, $\hat{\bB_i}$, which yields the following equivalent optimization:
\begin{align*}
     \max_{\mathbf{C}_i \in \calS_{\bC}} \min_{\beta \ge 0} \mathbf{C}_i^T \boldsymbol{\Sigma} \mathbf{g} \beta - \beta \|\mathbf{h}\| \|\boldsymbol{\Sigma}\mathbf{C}_i\| + \sigma_i\mathbf{C}_{ii}- \frac{\|\mathbf{C}_i\|^2}{4} 
\end{align*}
Now we use the square-root trick $\sqrt{x}:= \min_{\tau \ge 0} \frac{1}{\tau} + \tau x^2$ to obtain the following optimization
\begin{align*}
     \max_{\mathbf{C}_i \in \calS_{\bC}} \min_{\beta \ge 0} \max_{\tau \ge 0} & \;\mathbf{C}_i^T \boldsymbol{\Sigma} \mathbf{g} \beta - \frac{\beta  \|\mathbf{h}\|}{2} \left(\tau\|\boldsymbol{\Sigma}\mathbf{C}_i\|^2 + \frac{1}{\tau}\right) \\ &+ \sigma_i\mathbf{C}_{ii}- \frac{\|\mathbf{C}_i\|^2}{4}
\end{align*}
Rearranging terms yields
\begin{align}
    \max_{\mathbf{C}_i \in \calS_{\bC}}  \min_{\beta \ge 0} \max_{\tau \ge 0} - \frac{\beta\|\mathbf{h}\|}{2\tau} + \mathbf{C}_i^T \left(\beta \boldsymbol{\Sigma} \mathbf{g} + \sigma_i\mathbf{e}_i\right) \nonumber \\
    \quad - \frac{\beta\tau  \|\mathbf{h}\|}{2} \|\boldsymbol{\Sigma}\mathbf{C}_i\|^2 - \frac{\|\mathbf{C}_i\|^2}{4} \label{opt: scal_not}
\end{align}

Now we observe that the objective in \eqref{opt: scal_not} is convex-concave with respect to $\beta$ and $(\tau, \bC_i)$ respectively. Hence using Sion's minimax theorem, we may swap the order of the minimization and maximizations to obtain
\begin{align*}
    \min_{\beta \ge 0} \max_{\tau \ge 0} \max_{\mathbf{C}_i \in \calS_{\bC}}  - \frac{\beta\|\mathbf{h}\|}{2\tau} + \mathbf{C}_i^T \left(\beta \boldsymbol{\Sigma} \mathbf{g} + \sigma_i\mathbf{e}_i\right) \\
    \quad - \frac{\beta\tau  \|\mathbf{h}\|}{2} \|\boldsymbol{\Sigma}\mathbf{C}_i\|^2 - \frac{\|\mathbf{C}_i\|^2}{4}
\end{align*}
Performing the quadratic optimization over $\bC_i$ yields
\begin{align}\label{eq: obj_before_HW}
& \min_{\beta \ge 0} \max_{\tau \ge 0}  - \frac{\beta\|\mathbf{h}\|}{2\tau} \nonumber \\ 
& + \Bigl(\beta \boldsymbol{\Sigma} \mathbf{g} + \sigma_i\mathbf{e}_i\Bigr)^T\Bigl(\mathbf{I} + 2 \beta\tau \|\mathbf{h}\|\boldsymbol{\Sigma}^2\Bigr)^{-1} \Bigl(\beta \boldsymbol{\Sigma} \mathbf{g} + \sigma_i\mathbf{e}_i\Bigr) 
\end{align}

% Let $\mathbf{M} = \Bigl(\mathbf{I} + 2\beta\tau \|\mathbf{h}\|\boldsymbol{\Sigma}^2\Bigr)^{-1}$. Expanding the quadratic form:
% \begin{align*}
% &\Bigl(\beta \boldsymbol{\Sigma} \mathbf{g} + \sigma_i\mathbf{e}_i\Bigr)^T\mathbf{M}\Bigl(\beta \boldsymbol{\Sigma} \mathbf{g} + \sigma_i\mathbf{e}_i\Bigr) = \\
% & \beta^2 \mathbf{g}^T \boldsymbol{\Sigma} \mathbf{M} \boldsymbol{\Sigma} \mathbf{g} + 2\beta\sigma_i \mathbf{g}^T \boldsymbol{\Sigma} \mathbf{M} \mathbf{e}_i + \sigma_i^2 \mathbf{e}_i^T \mathbf{M} \mathbf{e}_i
% \end{align*}

% Since $\mathbf{g} \sim \mathcal{N}(0, \mathbf{I})$, by Hanson-Wright inequality:
% \begin{align*}
% &\mathbb{P}\Bigl(\Bigl|\mathbf{g}^T \boldsymbol{\Sigma} \mathbf{M} \boldsymbol{\Sigma} \mathbf{g} - \mathrm{tr}(\boldsymbol{\Sigma} \mathbf{M} \boldsymbol{\Sigma})\Bigr| > t\Bigr) \\
% &\leq 2\exp\left(-\frac{1}{2}\min\left(\frac{t^2}{4\|\boldsymbol{\Sigma} \mathbf{M} \boldsymbol{\Sigma}\|_F^2}, \frac{t}{2\|\boldsymbol{\Sigma} \mathbf{M} \boldsymbol{\Sigma}\|_{\text{op}}}\right)\right) \nonumber
% \end{align*}

% For the cross term, since $\mathbf{g} \sim \mathcal{N}(0, \mathbf{I})$ and $\boldsymbol{\Sigma} \mathbf{M} \mathbf{e}_i$ is deterministic:
% \begin{align*}
% \mathbb{P}\Bigl(\Bigl|\mathbf{g}^T \boldsymbol{\Sigma} \mathbf{M} \mathbf{e}_i\Bigr| > t\Bigr) \leq 2\exp\left(-\frac{t^2}{2\|\boldsymbol{\Sigma} \mathbf{M} \mathbf{e}_i\|^2}\right)
% \end{align*}

As $m \to \infty$, due to the Hanson-Wright inequality, \eqref{eq: obj_before_HW} converges in probability to:
\begin{align*}
\min_{\beta \ge 0} \max_{\tau \ge 0} &- \frac{\beta\|\mathbf{h}\|}{2\tau} +\beta^2 \mathrm{tr} \left(\boldsymbol{\Sigma}^2\Bigl(\mathbf{I} + 2 \beta\tau \|\mathbf{h}\|\boldsymbol{\Sigma}^2\Bigr)^{-1}\right) \nonumber\\
&+ \frac{\sigma_i^2}{1 + 2\beta\tau\|\mathbf{h}\|\sigma_i^2} \nonumber
\end{align*}
Which can be rewritten as:
\begin{align}\label{eq: right_after_HW}
\min_{\beta \ge 0} \max_{\tau \ge 0} - \frac{\beta\|\mathbf{h}\|}{2\tau} +  \sum_j \frac{\sigma_j^2(\beta^2 + \delta_{ij})}{1 + 2\beta\tau\|\mathbf{h}\|\sigma_j^2} 
\end{align}

Setting derivatives by $\tau$ and $\beta$ to zero:
\begin{align*}
&\frac{\beta\|\mathbf{h}\|}{2\tau^2} -  \sum_j \frac{2\sigma_j^4\beta\|\mathbf{h}\|(\beta^2 + \delta_{ij})}{(1 + 2\beta\tau\|\mathbf{h}\|\sigma_j^2)^2} = 0, \\
&- \frac{\|\mathbf{h}\|}{2\tau} + \sum_j \sigma_j^2 \frac{2\beta(1 + 2\beta\tau\|\mathbf{h}\|\sigma_j^2)}{(1 + 2\beta\tau\|\mathbf{h}\|\sigma_j^2)^2} \nonumber\\
&\quad - \sum_j \sigma_j^2\frac{2\tau \|\mathbf{h}\|\sigma_j^2(\beta^2 + \delta_{ij})}{(1 + 2\beta\tau\|\mathbf{h}\|\sigma_j^2)^2} = 0 \nonumber
\end{align*}

Simplifying both equations we arrive at:
\begin{align*}
\frac{1}{2\tau^2} &=  \sum_j \frac{2\sigma_j^4(\beta^2 + \delta_{ij})}{(1 + 2\beta\tau\|\mathbf{h}\|\sigma_j^2)^2} \\
\frac{\|\mathbf{h}\|}{2\tau}  &= \sum_j \sigma_j^2 \frac{2\beta + 2\beta^2\tau\|\mathbf{h}\|\sigma_j^2 - 2\tau \|\mathbf{h}\|\sigma_j^2 \delta_{ij}}{(1 + 2\theta\|\mathbf{h}\|\sigma_j^2)^2} \nonumber
\end{align*}

Introducing $\tilde{\theta}= \beta\tau$ we get:
\begin{align*}
1 &=  \sum_j \frac{4\sigma_j^4(\tilde{\theta}^2 + \delta_{ij}\tau^2)}{(1 + 2\tilde{\theta}\|\mathbf{h}\|\sigma_j^2)^2} \\
\|\mathbf{h}\| &= \sum_j 4 \sigma_j^2 \frac{\tilde{\theta} + \tilde{\theta}^2\|\mathbf{h}\|\sigma_j^2 - \tau^2 \|\mathbf{h}\|\sigma_j^2 \delta_{ij}}{(1 + 2\tilde{\theta}\|\mathbf{h}\|\sigma_j^2)^2} \nonumber
\end{align*}

Replacing the penultimate equation above by the the sum of the last equation with the penultimate one multiplied by $\|\mathbf{h}\|$ and canceling out $\tilde{\theta}$ we arrive at:
% \begin{align*}
% 2\|\mathbf{h}\| &= \sum_j 4 \sigma_j^2 \frac{\tilde{\theta} + 2\tilde{\theta}^2\|\mathbf{h}\|\sigma_j^2}{(1 + 2\tilde{\theta}\|\mathbf{h}\|\sigma_j^2)^2}\nonumber
% \end{align*}
%  leads to the following equation:
\begin{align}\label{eq: theta}
\sum_j 2 \sigma_j^2 \frac{\tilde{\theta}}{1 + 2\tilde{\theta}\|\mathbf{h}\|\sigma_j^2} = \|\mathbf{h}\| 
\end{align}
The value of \eqref{eq: right_after_HW} simplifies using \eqref{eq: theta} as:
\begin{align*}
&- \frac{\tilde{\theta}\|\mathbf{h}\|}{2\tau^2} +  \sum_j \frac{\sigma_j^2(\beta^2 + \delta_{ij})}{1 + 2\tilde{\theta}\|\mathbf{h}\|\sigma_j^2} = \frac{\sigma_i^2}{1 + 2\tilde{\theta}\|\mathbf{h}\|\sigma_i^2}
\end{align*}
The value of \eqref{eq: right_after_HW} simplifies as follows using \eqref{eq: theta}:
\begin{align*}
&- \frac{\tilde{\theta}\|\mathbf{h}\|}{2\tau^2} +  \sum_j \frac{\sigma_j^2(\beta^2 + \delta_{ij})}{1 + 2\tilde{\theta}\|\mathbf{h}\|\sigma_j^2} = - \frac{\tilde{\theta}\|\mathbf{h}\|}{2\tau^2} +  \sum_j \frac{\sigma_j^2(\frac{\tilde{\theta}^2}{\tau^2} + \delta_{ij})}{1 + 2\tilde{\theta}\|\mathbf{h}\|\sigma_j^2 } \\
&=   \frac{\tilde{\theta}}{2\tau^2} \left( \sum_j \frac{2\tilde{\theta} \sigma_j^2}{1 + 2\tilde{\theta}\|\mathbf{h}\|\sigma_j^2 } -\|\mathbf{h}\| \right) +   \frac{\sigma_i^2}{1 + 2\tilde{\theta}\|\mathbf{h}\|\sigma_i^2}\\ & = \frac{\sigma_i^2}{1 + 2\tilde{\theta}\|\mathbf{h}\|\sigma_i^2}
\end{align*}
We would also like to slightly simplify \eqref{eq: theta}
\begin{align*}
& \|\mathbf{h}\|^2 = \sum_j 2 \sigma_j^2 \frac{\tilde{\theta}\|\mathbf{h}\| }{1 + 2\tilde{\theta}\|\mathbf{h}\|\sigma_j^2} = \\
&  \sum_j 1 - \frac{1}{1 + 2\tilde{\theta}\|\mathbf{h}\|\sigma_j^2} = m - \sum_{j=1}^m \frac{1}{1 + 2\tilde{\theta}\|\mathbf{h}\|\sigma_j^2}
\end{align*}
% Thus
% Also note that \eqref{eq: theta} can be simplified as: 
% \begin{align*}
% m - \|\mathbf{h}\|^2 = \sum_{j=1}^m \frac{1}{1 + 2\tilde{\theta}\|\mathbf{h}\|\sigma_j^2}
% \end{align*}
Denoting $\theta_{\bh} := 2\theta \|\mathbf{h}\|$ finally yields the desired expression $\frac{\sigma_i^2}{1 +\theta_{\bh} \sigma_i^2}$ for \eqref{eq: obj_before_HW}, where $\theta_{\bh}$ is defined through
\begin{align*}
\sum_{j=1}^m \frac{1}{1 + \theta_\bh \sigma_j^2} = m - \|\mathbf{h}\|^2 
\end{align*}

\end{proof}

\begin{lemma}\label{lm: phi_hw_exp}
 Define $\theta(t)$ via 
$$\sum_{j=1}^m \frac{1}{1 + \theta(t) \sigma_j^2} = m - t$$
Let $\phi^{HW}_i(\mathbf{h}) = \frac{\sigma_i^2}{1 +\theta_{\bh} \sigma_i^2}$ and $c_i = \frac{\sigma_i^2}{1 +\theta\sigma_i^2}, $ where $\theta_{\bh} = \theta(\|\bh\|)$ and $\theta = \theta(k)$
Then we have that
\begin{align*}
    \bbE \phi^{HW}_i(\mathbf{h}) \sim c_i
\end{align*}
\end{lemma}

\begin{proof}
    The proof follows from standard Gaussian concentration and Lipschitness of $\theta(t)$.
\end{proof}

\subsection{Propositions}

We prove only Proposition \ref{prop: upp_bnd} here due to the lack of space, but the proof of Proposition \ref{prop: upp_bnd_filt} is almost identical.  

\begin{proof}[Proof of Proposition \ref{prop: upp_bnd}]
The lower bound is a well-lnown inequality for any rank-$k$ approximation of $\bA$. 
To prove the upper bound,  suffices to show that 
\begin{align*}
    \theta^{C}_*:= \frac{k}{C\sum_{i = r + 1}^m \sigma_i^2} \le \theta
\end{align*}
Since $f(x) = \sum \frac{1}{1 + x \sigma_i^2}$ decreases in $x$, it is enough to demonstrate the following: 
\begin{align}\label{eq: f_ineq}
    f(\theta^C_*)  \ge m - k 
\end{align}
To do so, it suffices to prove the following for any  $S>0$: 
    $$\inf_{ \sigma_{r+1}^2 + \dots + \sigma_d^2 = S} f(\theta^C_*) \ge f(\theta) = m - k  $$
Note that $f(\theta^C_*)$ is decreasing in $\sigma_1,\dots, \sigma_r$
Since $\sigma_1,\dots,\sigma_r$ are unconstrained, the infimum is attained when $\sigma_1,\dots,\sigma_r \to \infty$ and we are left with $$\inf_{ \sigma_{r+1}^2 + \dots + \sigma_d^2 = S}  \sum_{i = r + 1}^d \frac{1}{1 + \frac{k\sigma_i^2}{CS}}$$ 
Now, using Jensen's inequality, we have:
\begin{align*}
    & f(\theta_*) \ge \frac{d-r}{1 + \frac{k}{C(d-r)}} = \frac{d-r}{1 + \frac{k}{\frac{(d-k)k}{(d-r)(k-r)}(d-r)}} \\
    & = \frac{d-r}{1 + \frac{k}{\frac{(d-k)k}{(d-r)(k-r)}(d-r)}} = \frac{d-r}{1 + \frac{k-r}{d-k}} = d - k
\end{align*}
\end{proof}

\subsection{Proof of Corollary \ref{cor: ensemble}}
\textbf{Bilevel Ensembles.} Plugging-in the assumption into the result of Theorem \ref{thm: err} for $\tilde{\Theta}$ yields
\begin{align*}
    \frac{ra^2}{\tilde{\Theta}+ k a^2} + \frac{(m-r) b^2}{\tilde{\Theta} + k b^2} = 1
\end{align*}
Performing some algebraic manipulations imply that the error is:
\begin{align*}
     & \tilde{\Theta} =
    \frac{1}{2} \Bigl((r-k) a^2 + (m - r- k) b^2 \\& + \sqrt{((r-k) a^2 + (m - r- k) b^2)^2 + 4 a^2 b ^2 (m-k)} \Bigr)
\end{align*}
As for an arbitrary filter $\chi(\cdot)$,
\begin{align} \label{eq: arb_filt_theta}
    \frac{r}{1 + \theta_0 \chi(a)^2} + \frac{m-r}{1 + \theta_0 \chi(b)^2} = m - k
\end{align}
Solving for $\theta_0$ yields
\begin{align*}
    \theta_0 &= \frac{1}{2(m-k) \chi(a)^2 \chi(b)^2 } \\ &\cdot \biggl( (k-r) \chi(a)^2 + (k + r - m) \chi(b)^2 \\ &+ \Bigl(\Bigl((k-r) \chi(a)^2 + (k + r - m) \chi(b)^2 \Bigr)^2 \\ &+ 4 k (m-k) \chi(a)^2 \chi(b)^2\Bigr)^{1/2}
    \biggr) 
\end{align*}
By plugging $\theta_0$ into the result of Theorem \ref{thm: err_filt}, we obtain the expression for the error of arbitrary polynomial filter. \\
\textbf{Structure of the optimal filter.} Using \eqref{eq: arb_filt_theta}, we observe that
\begin{align*}
    (m-k)\cdot (1 + \theta_0 \chi(a)^2)\cdot (1 + \theta_0 \chi(b)^2) \\ = r (1 + \theta_0 \chi(b)^2) + (m-r) (1 + \theta_0 \chi(a)^2)
\end{align*}
Furthermore, for the error
\begin{align*}
    &\tilde{\Theta} = \frac{r a^2}{1 + \theta_0 \chi(a)^2} + \frac{(m-r) b^2}{1 + \theta_0 \chi(b)^2} \\
    &= (m-k) \cdot \Bigl(\frac{r (1 + \theta_0 \chi(b)^2)}{r (1 + \theta_0 \chi(b)^2) + (m-r) (1 + \theta_0 \chi(a)^2)}\cdot a^2 \\ &+ \frac{(m-r) (1 + \theta_0 \chi(a)^2)}{r (1 + \theta_0 \chi(b)^2) + (m-r) (1 + \theta_0 \chi(a)^2)} \cdot b^2\Bigr)
\end{align*}
It can be seen that the error is a convex combination of $a^2$ and $b^2$ and the best strategy is to maximize the ratio $\frac{1 + \theta_0 \chi(a)^2}{1 + \theta_0 \chi(b)^2}$ assuming $a > b$. Note that by \eqref{eq: arb_filt_theta}, it suffices to make $1 + \theta_0 \chi(a)^2$ large which follows when $1 + \theta_0 \chi(b)^2 \approx \frac{m-r}{m-k}$, which is possible when $k > r$; this implies for $\theta$
\begin{align*}
    \theta_0 \approx \frac{1}{\chi(b)^2} \cdot \frac{k-r}{m-k}
\end{align*}
Which implies $1 + \theta_0 \chi(a)^2 \approx 1 + \frac{k-r}{m-k} \cdot \frac{\chi(a)^2}{\chi(b)^2} $ and it suffices to take $\frac{\chi(a)^2}{\chi(b)^2}$ to be large. Hence the higher the degree of the polynomial filter, the smaller error would be. The best achievable error is $(m-k) b^2$ which matches the fundamental lower bound in Proposition \ref{prop: upp_bnd}.

\textbf{Power Law Ensembles. } We use the result in \cite{simon2023more} and \cite{ildiz2024high} which follows by an integral approximation of the series in the regime $m \rightarrow \infty$.
\begin{align*}
    1 = \sum_{i=1}^m \frac{i^{-2\alpha}}{\tilde{\Theta} + k i^{-2\alpha}} = \frac{1}{k} \sum_{i=1}^m \frac{i^{-2\alpha}}{\frac{\tilde{\Theta}}{k} + i^{-2\alpha}} 
\end{align*}
This implies for the error from \cite{ildiz2024high}
\begin{align*}
     \tilde{\Theta} \sim \frac{1}{\Bigl(k \sinc(\frac{\pi}{2\alpha} )\Bigr)^{2\alpha}} \Bigl(1 + O(\frac{1}{k})\Bigr)
\end{align*}
Now we consider an arbitrary polynomial filter $\chi(t) := t^p$. First, to evaluate $\theta$, we consider the following equation:
\begin{align*}
    k = \sum_{i=1}^m \frac{\theta_0 i^{-2\alpha p}}{1 + \theta_0 i^{-2\alpha p}} 
\end{align*}
Which implies from \cite{ildiz2024high}
\begin{align*}
    \theta_0 = \Bigl(k \sinc(\frac{\pi}{2\alpha p} )\Bigr)^{2\alpha p} \Bigl(1 + O(\frac{1}{k})\Bigr)
\end{align*}
As for the error from Theorem \ref{thm: err_filt} we have
\begin{align*}
    \tilde{\Theta} = \sum_{i=1}^\infty \frac{i^{-2\alpha}}{1 + \theta_0 i^{-2\alpha p}}  \approx  \sum_{i=1}^\infty \frac{i^{-2\alpha}}{1 +  \Bigl(\frac{k}{i} \sinc(\frac{\pi}{2\alpha p} )\Bigr)^{2\alpha p} }
\end{align*}
Now when $p \rightarrow \infty$, we observe that
\begin{align*}
    \sinc^{2\alpha p}(\frac{\pi}{2\alpha p} ) \approx 1 + o(1)
\end{align*}
Thus for each in entry in the sum above we have
\begin{align*}
    (\frac{k}{i})^{2\alpha p} \rightarrow \begin{cases}
        \infty & i < k\\
        1 & i = k \\
        0 & i > k
    \end{cases}
\end{align*}
Therefore, for the error
\begin{align*}
    & \sum_{i=1}^{\infty} \frac{i^{-2\alpha}}{1 + \theta_0 i^{-2\alpha p}} \approx \frac{k^{-2\alpha}}{1 +  \Bigl(\sinc(\frac{\pi}{2\alpha p} )\Bigr)^{2\alpha p} } \\ & + \sum_{i=1}^{k-1} \frac{i^{-2\alpha}}{1 +  \Bigl(\frac{k}{i} \sinc(\frac{\pi}{2\alpha p} )\Bigr)^{2\alpha p} }  + \sum_{i=k+1}^{\infty} \frac{i^{-2\alpha}}{1 +  \Bigl(\frac{k}{i} \sinc(\frac{\pi}{2\alpha p} )\Bigr)^{2\alpha p} } \\ 
   &  \approx \frac{k^{-2\alpha}}{2} + o(1) + \sum_{i=k+1}^{\infty} i^{-2\alpha} \approx \sum_{i=k+1}^{\infty} i^{-2\alpha}
\end{align*}
Which concludes the proof.

\bibliographystyle{plainnat}
\bibliography{main}

\end{document}

%% file: symbols.tex
\newcommand{\bSigma}{\bm{\Sigma}}
\newcommand{\bOmega}{\bm{\Omega}}

\newcommand{\tTheta}{\tilde{\Theta}}
\newcommand{\rk}[1]{\operatorname{rank}(#1)}

\newcommand{\bA}{\mathbf{A}}
\newcommand{\bB}{\mathbf{B}}
\newcommand{\bC}{\mathbf{C}}

\newcommand{\bQ}{\mathbf{Q}}
\newcommand{\bR}{\mathbf{R}}

\newcommand{\bU}{\mathbf{U}}
\newcommand{\bV}{\mathbf{V}}

\newcommand{\bX}{\mathbf{X}}
\newcommand{\bY}{\mathbf{Y}}
\newcommand{\bZ}{\mathbf{Z}}

\newcommand{\bh}{\mathbf{h}}

\newcommand{\bbE}{\mathbb{E}}
\newcommand{\bbR}{\mathbb{R}}

\newcommand{\calS}{\mathcal{S}}

\newcommand{\sinc}{\text{sinc}}